\documentclass{amsart}
\usepackage[utf8]{inputenc}
\usepackage[noadjust]{cite}
\usepackage{amsfonts}
\usepackage{amssymb}
\usepackage{amsthm}
\usepackage{enumitem}
\usepackage{amsmath}
\usepackage{csquotes}
\usepackage{aliascnt}
\usepackage{mathtools}
\usepackage[unicode,bookmarksopen]{hyperref}

\newtheorem{theorem}{Theorem}[section]

\newcommand{\MakeTheoremAndCounter}[2]{\newaliascnt{#1}{theorem}
\newtheorem{#1}[#1]{#2}
\aliascntresetthe{#1}
\expandafter\providecommand\csname#1autorefname\endcsname{#2}}

\MakeTheoremAndCounter{lemma}{Lemma}
\MakeTheoremAndCounter{claim}{Claim}
\MakeTheoremAndCounter{example}{Example}
\MakeTheoremAndCounter{fact}{Fact}
\MakeTheoremAndCounter{question}{Question}
\MakeTheoremAndCounter{corollary}{Corollary}
\newtheorem*{restatelargemain}{\autoref{largemain}}

\theoremstyle{definition}
\MakeTheoremAndCounter{definition}{Definition}
\newtheorem*{remark}{Remark}

\newlist{multistmt}{enumerate}{1}
\setlist[multistmt]{label={\upshape(\roman*)}, nosep}

\DeclareMathOperator*{\diam}{diam}
\DeclareMathOperator*{\supp}{supp}
\DeclareMathOperator*{\gap}{gap}
\DeclareMathOperator*{\tallinf}{inf\vphantom{p}}
\DeclareMathOperator*{\Int}{int}
\newcommand{\ap}[2]{#1[#2]}
\newcommand{\ad}[2]{[#1]^{#2}}

\newcommand{\titleofthepaper}{Substituting the typical compact sets into a power series}
\hypersetup{pdftitle=\titleofthepaper}
\begin{document}
\title{\titleofthepaper}

\author{Don\'at Nagy}
\address{E\"otv\"os Lor\'and University, Institute of Mathematics, P\'azm\'any P\'eter
s. 1/c, 1117 Budapest, Hungary}
\email{m1nagdon@gmail.com}

\thanks{The author was supported by the National Research, Development and Innovation Office -- NKFIH, grants no.~104178 and 124749.}
\begin{abstract}
The Minkowski sum and Minkowski product can be considered as the addition and multiplication of subsets of $\mathbb{R}$. If we consider a compact subset $K\subseteq[0,1]$ and a power series $f$ which is absolutely convergent on $[0,1]$, then we may use these operations and the natural topology of the space of compact sets to substitute the compact set $K$ into the power series $f$. Changhao Chen studied this kind of substitution in the special case of polynomials and showed that if we substitute the typical compact set $K\subseteq [0,1]$ into a polynomial, we get a set of Hausdorff dimension 0. We generalize this result and show that the situation is the same for power series where the coefficients converge to zero quickly. On the other hand we also show a large class of power series where the result of the substitution has Hausdorff dimension one.
\end{abstract}

\maketitle

\tableofcontents

\section{Introduction}

We may use the Minkowski sum and the Minkowski product to add or multiply two subsets of $\mathbb{R}$:
\[A+B = \{a+b : a\in A, b\in B\},\quad A\cdot B = \{a\cdot b : a\in A, b\in B\}\qquad (A, B \subseteq \mathbb{R}).\]
Iterated application of these operations (and the simple multiplication $c\cdot A = \{c\cdot a : a\in A\}$ for $A\subseteq \mathbb{R}$ and $c\in \mathbb{R}$) allow us to substitute a subset of $\mathbb{R}$ into a polynomial with real coefficients. In \cite[Theorem 1.5]{Chen} Changhao Chen proves the following:

\begin{theorem}[Changhao Chen]\label{thmchen}
Consider a fixed polynomial $p$ (with real coefficients). The typical nonempty compact set $K\subseteq [0,1]$ satisfies that if we substitute $K$ into $p$, then we get a set of Hausdorff dimension 0.
\end{theorem}

Recall that we say that a property $\mathcal{P}$ is true for the \emph{typical nonempty compact subset} of a space $X$ when
\[\{K\in \mathcal{K}(X) : K \text{ does not satisfy }\mathcal{P}\}\quad\text{is meager in $\mathcal{K}(X)$,}\]
where $\mathcal{K}(X)$ is the space of nonempty compact subsets of $X$. We include the definitions of the classical notions of \emph{meager sets} and \emph{Hausdorff dimension} in \autoref{prelim}, where we also state some well-known facts about the topology of $\mathcal{K}(X)$. 

Chen also noticed that it is possible to generalize this and substitute a compact set $A\subseteq \mathbb{R}$ into a power series $f$, assuming that $f$ satisfies some simple convergence conditions. (The result of the substitution is simply the limit of the partial sums w.r.t.\ the topology of $\mathcal{K}(X)$.)

We generalize \autoref{thmchen} in this direction and prove \autoref{zerodimimg}, which states that if the coefficients of the power series are small enough (i.e.\ they converge to 0 quickly), then the image of the typical compact set is zero-dimensional.

On the other hand, we also prove \autoref{largemain} which states that if the coefficients are large enough (i.e.\ their convergence to 0 is slow enough), then the image of the typical compact set will contain an interval (and this means that its Hausdorff dimension is one).

This problem is also similar to the classical and well-studied area of Bernoulli convolutions. The Bernoulli convolution $\nu_\lambda$ corresponding to the parameter $0<\lambda<1$ is defined as the distribution of the random variable
\[\sum_{i=0}^\infty X_i \cdot \lambda^i,\]
where $X_0$, $X_1$, \ldots\ are independent random variables taking the values $\pm 1$ with equal probability. For the definitions, questions and results in this area, see e.g. \cite{PSS}.

Notice that the infinite sum of random variables (which corresponds to an infinite convolution of measures) is a close analogue to infinite Minkowski sums of (compact) sets. The only difference is that while our results simply consider the Minkowski sums of unstructured sets, the definition of Bernoulli convolutions considers sets equipped with probability measures and maintains this additional structure through the Minkowski summation.

With this additional structure, results about Bernoulli convolutions may speak of the Hausdorff dimension of the measure $\nu_\lambda$ (see \cite[section 4]{PSS} for the definition and some results), which may be interesting even in cases when it is trivial that $\dim_H \supp \nu_\lambda = 1$. On the other hand, our results take infinite sums of sets that are small and simple in certain senses, but usually not as simple as the set $\{-\lambda^i, \lambda^i\}$ that appears in the construction of $\nu_\lambda$.

\section{Preliminaries}\label{prelim}

Recall that a set is \emph{nowhere dense} if its closure has empty interior and a set is \emph{meager} if it is a countable union of nowhere dense sets. This notion is customarily used to formally state that some set is small and negligible. For example we say that a property $\mathcal{P}$ is satisfied by the \emph{the typical element} of a topological space $X$ when
\[\{x\in X : x \text{ does not satisfy }\mathcal{P}\}\quad\text{is meager in $X$.}\]
For more information about these see e.g.\ \cite[section 8]{Ke}.

Also recall that if $(X, d)$ is a metric space (e.g. $X=[0,1]$), then its \emph{nonempty} compact subsets also form a metric space $(\mathcal{K}(X), d_\mathrm{H})$, where $d_\mathrm{H}$ is the Hausdorff metric defined by
\[d_\mathrm{H}(K,L)=\max\left\{\sup_{k\in K}\tallinf_{\ell\in L}d(k,\ell),\,\sup_{\ell\in L}\tallinf_{k\in K}d(k,\ell)\right\}.\]
It is well known that if $X$ is separable, complete or compact, then $\mathcal{K}(X)$ is respectively separable, complete or compact. For the proofs of these results and more facts about this space see e.g.\ \cite[subsection 4.F]{Ke}.

\begin{remark}
In \cite{Ke}, $\mathcal{K}(X)$ is defined as the space of all compact subsets of $X$, but we defined $\mathcal{K}(X)$ as the space of all \emph{nonempty} compact subsets of $X$. This modified definition is frequently used in the literature, as otherwise there are many situations where the trivial case of $\emptyset$ is exceptional. (For example the definition of $d_\mathrm{H}$ must be extended by the cases $d_\mathrm{H}(\emptyset, \emptyset) = 0$ and $d_\mathrm{H}(\emptyset, K) = 1$ for $K \neq \emptyset$.)
\end{remark}

We also use the classical notion of the Hausdorff dimension. The Hausdorff dimension of a set $F\subseteq \mathbb{R}$ is defined as
\[\dim_\mathrm{H} F = \inf\{s \ge 0 : \mathcal{H}^s(F) = 0\} = \sup\{s : \mathcal{H}^s(F) = \infty\}\]
where we use the convention that $\sup\emptyset = 0$ and $\mathcal{H}^s(F)$ denotes the $s$-dimensional Hausdorff measure
\[\mathcal{H}^s(F) = \lim_{\delta\to 0} \inf\left\{\sum_{i=1}^\infty \diam( U_i)^s : F\subseteq \bigcup_{i=1}^\infty U_i \text{ and }\diam(U_i) <\delta\right\}.\]

To calculate the Hausdorff dimension, we introduce the simpler quantity
\[\mathcal{H}^s_\infty(F) = \inf\left\{\sum_{i=1}^\infty \diam(U_i)^s : F\subseteq \bigcup_{i=1}^\infty U_i\right\},\]
which is usually called the $s$-dimensional Hausdorff capacity of $F$. The elementary observation \cite[Proposition 5.1.5]{AH} states that $\mathcal{H}^s_\infty(F) =0$ if and only if $\mathcal{H}^s(F)=0$ and this immiediately implies that
\[\dim_\mathrm{H} F = \inf\{s \ge 0 : \mathcal{H}^s_\infty(F) = 0\}.\]
(The usual definition of $\dim_{\mathrm{H}}$ is stated with $\mathcal{H}^s$ because $\mathcal{H}^s_\infty$ lacks some nice properties of $\mathcal{H}^s$, e.g.\ the intervals are not measurable with respect to $\mathcal{H}^s_\infty$.)

For more information about the Hausdorff dimension see e.g. \cite[section 2]{Fa}.

\section{Notations}

As the notation $A^n$ usually denotes the set $\{a^n : a\in A\}$, we introduce a new piece of notation to denote the iterated application of the Minkowski product:

\begin{definition}
For a set $A\subseteq \mathbb{R}$ and $n\in \mathbb{N}$, let $\ad{A}{n}$ denote the Minkowski product of $n$ copies of $A$:
\[\ad{A}{n} = \left\{\prod_{i=0}^{n-1} a_i : a_i \in A \text{ for }0\le i < n\right\}\]
In particular, we define $\ad{A}{0}=\{1\}$ for any set $A\subseteq \mathbb{R}$.
\end{definition}

We also use a similar notation for substituting a compact set into a power series:

\begin{definition}
If $K\subseteq [0,1]$ is compact and $f(x)=\sum_{i=0}^\infty a_i x^i$ is a power series which converges absolutely on $[0,1]$, then let
\[\ap{f}{K}=\sum_{i=0}^\infty a_i \ad{K}{i}\]
where the summation is performed using the Minkowski sum and convergence w.r.t.\ the Hausdorff distance.
\end{definition}

\section{Lengths of intervals}

For an interval $I\subseteq \mathbb{R}$, we use $\diam(I)$ to denote the length (i.e. diameter) of the interval $I$.

\begin{fact}\label{diamtriv}
Assume that $I, J \subseteq \mathbb{R}$ are nonempty intervals and $t\in \mathbb{R}$. Then it is easy to verify the following:
\begin{multistmt}
\item $\diam(I+t) = \diam(I)$
\item $\diam(tI) = \lvert t\rvert \cdot \diam(I)$
\item $\diam(I+J) = \diam(I) + \diam(J)$
\end{multistmt}
\end{fact}

For the length of the Minkowski product of intervals, we will only need the following simple inequality:

\begin{lemma}\label{diamprod}
If $I, J \subseteq [0,1]$ are intervals, then
\[\diam(I\cdot J)\le \diam(I) + \diam(J).\]
\end{lemma}

\begin{proof}
Let $a = \inf I$, $b = \sup I$, $c = \inf J$ and $d=\sup J$ denote the endpoints of these intervals. Then $\diam(I) + \diam(J) = (b-a) + (d-c)$. On the other hand, $\inf (I\cdot J) = ac$ and $\sup(I\cdot J) = bd$ and therefore $\diam(I\cdot J) = bd-ac$. Therefore 
$\diam(I) + \diam(J) -\diam(I\cdot J) = (b-a) + (d-c) - (bd-ac)= (1-a)(1-c) - (1-b)(1-d)$
and this is clearly nonnegative, as $0\le 1-b < 1-a$ and $0\le 1-d < 1-c$.
\end{proof}

\section{The case of small images}

\begin{theorem}\label{zerodimimg}
Consider a power series $f(x) = \sum_{i=0}^\infty a_i x^i$ which converges absolutely on $[0,1]$ and satisfies that $\frac{\log \lvert a_i\rvert}{i\log i}$ converges to $-\infty$. Then the typical compact set $K\subseteq [0,1]$ satisfies that $\dim_\mathrm{H} \ap{f}{K} = 0$.
\end{theorem}

Note that in this paper $\log$ denotes the natural (base $e$) logarithm and we use the convention that $\log 0 = -\infty$. Also note that the absolute convergence of $f$ implies that for all but finitely many $n\in\mathbb{N}$, $\lvert a_i\rvert < 1$ and $\log \lvert a_i\rvert < 0$.

\begin{proof}
For $n\in\mathbb{N}$, let $R_n=\sum_{i=n}^\infty \lvert a_i\rvert$.

To simplify some technical details, let us assume that $R_n >0$ for all $n\in\mathbb{N}$. (We may make this assumption, because otherwise $f$ is a polynomial and \cite[Theorem 1.5]{Chen} already states that $\dim_\mathrm{H} \ap{f}{K}= 0$.)

For $K\in \mathcal{K}([0,1])$ and $n\in \mathbb{N}$, it is clear from the definition that
\begin{equation}\label{finiteexp}
\ap{f}{K} \subseteq \sum_{i=0}^{n-1} a_i \ad{K}{i} + [-R_n, R_n].
\end{equation}
We will use this to show that $\mathcal{H}^s_\infty(\ap{f}{K})=0$ for any $s>0$.

Notice that $R_n=\sum_{i=n}^\infty \lvert a_i\rvert$ is not much larger than its first term, $\lvert a_n\rvert$ (as $a_i$ converges to zero quickly). This observation can be formalized to prove the following claim:
\begin{claim}
$\displaystyle \frac{\log R_n}{n\log n}\longrightarrow -\infty$.
\end{claim}

\begin{proof}
If $K\ge 1$, then there exists an index $n_0\in\mathbb{N}$ such that if $i\ge n_0$, then $\frac{\log\lvert a_i\rvert}{i\log i}<-K$. This means that for all $n\ge \max\{n_0,2\}$,
\begin{align*}
R_n & = \sum_{i=n}^\infty \lvert a_i\rvert <\sum_{i=n}^\infty e^{-Ki\log i}\le e^{-Kn\log n} \cdot \sum_{i=0}^\infty e^{-K i \log n}\\
&=\frac{e^{-Kn\log n}}{1-n^{-K}}<2e^{-K n\log n},
\end{align*}
using $K\ge 1$ and $n\ge 2$ in the last inequality. This shows that $\frac{\log R_n}{n\log n} \longrightarrow -\infty$.
\end{proof}

It is clearly possible to select a sequence $(M_n)_{n\in\mathbb{N}}$ of positive integers which (slowly) converges to infinity but satisfies that
\begin{equation}\label{mnbound}
M_n\le n+1\quad\text{for each $n\in\mathbb{N}$}
\end{equation}
and
\begin{equation}\label{mnchoice}
\text{for any $s>0$,}\quad\lim_{n\to\infty} \left(M_n + s\cdot \frac{\log (4 R_n)}{n\log (2n)}\right) = -\infty.
\end{equation}

Fix a sequence $(\varepsilon_n)_{n\in\mathbb{N}}$ such that for all $n\in\mathbb{N}$,
\begin{equation}\label{epsnchoice}
0<\varepsilon_n < \min\left\{\frac{R_n}{nR_0},  \frac{1}{2M_n}\right\}.
\end{equation}
(Here $\varepsilon_n <\frac{R_n}{n R_0}$ is the important upper bound, while the assumption that $\varepsilon_n<\frac{1}{2M_n}$ is just a technical detail.)

For $n\in\mathbb{N}$ consider the sets
\[D_n = \left\{ \frac{2j+1}{2 M_n} : 0\le j < M_n\right\}\text{,}\quad \mathcal{D}_n = \{K\subseteq D_n : K \neq \emptyset\},\]
\[\mathcal{U}_n = B_{\varepsilon_n}(\mathcal{D}_n) = \left\{ K \in \mathcal{K}([0,1]) : \exists\, D\in \mathcal{D}_n \big(d_\mathrm{H}(K,D)<\varepsilon_n\big)\right\}\]
\[\text{and}\quad \mathcal{G} = \bigcap_{n_0\in\mathbb{N}}\bigcup_{n\ge n_0} \mathcal{U}_n.\]

\begin{claim}
$\mathcal{G}$ is a dense $G_\delta$ set in $\mathcal{K}([0,1])$.
\end{claim}

\begin{proof}
It is enough to prove that $\bigcup_{n\ge n_0} \mathcal{U}_n$ is dense and open for each $n_0\in\mathbb{N}$. The openness of this set immediately follows from the fact that $\mathcal{U}_n$ is open for each $n\in \mathbb{N}$. To show that this set is dense, it is enough to show that even $\bigcup_{n\ge n_0} \mathcal{D}_n$ is dense for each $n_0\in\mathbb{N}$.

Fix arbitrary $K\in \mathcal{K}([0,1])$ and $\varepsilon>0$. Choose an index $n\ge n_0$ which satisfies that $\frac{1}{2M_n} <\varepsilon$ (such $n$ exists because $M_n$ is unbounded). Then the set
\[D= \left\{ \frac{2j+1}{2 M_n} : 0\le j < M_n, \left[\frac{j}{M_n}, \frac{j+1}{M_n}\right]\cap K\neq\emptyset\right\}\]
is a nonempty subset of $D_n$ and it is easy to verify that 
\[d_\mathrm{H}(K, D) \le \frac{1}{2M_n}<\varepsilon.\qedhere\]
\end{proof}

To prove \autoref{zerodimimg}, we will fix an arbitrary $s>0$ and a compact set $K\in\mathcal{G}$ and show that $\mathcal{H}^s_\infty(\ap{f}{K})=0$.

Assume that $n\in\mathbb{N}$ is an index such that $K\in\mathcal{U}_n$. (There are infinitely many indices satisfying this, because $K\in\mathcal{G}$.) This assumption means that
\[d_\mathrm{H}(K, D) <\varepsilon_n\quad\text{for some nonempty }D\subseteq D_n.\]
Consider the intervals 
\[I_{\ell}=\left]\frac{2\ell+1}{2M_n}-\varepsilon_n, \frac{2\ell+1}{2M_n}+\varepsilon_n\right[\text{ for }0\le \ell <M_n.\]
\begin{remark} As we have assumed that $\varepsilon_n <\frac{1}{2M_n}$, here $I_\ell\subseteq [0,1]$ for $0\le \ell <M_n$. 
\end{remark}
As $D\subseteq D_n$ and $K$ is near $D$ (in the Hausdorff distance), it is easy to see that
\[K\subseteq \bigcup_{\ell=0}^{M_n-1} I_\ell = B_{\varepsilon_n}(D_n).\]
We will use this cover of $K$ to select a system of intervals which covers $\ap{f}{K}$. To do this, we will first prove the following claim:

\begin{claim}\label{coverprod}
If $i$ is a positive integer, then there exists a system $\mathcal{J}_i$ of intervals such that
\begin{multistmt}
\item $\ad{K}{i}\subseteq \bigcup_{J\in \mathcal{J}_i} J$,
\item for each $J\in\mathcal{J}_i$, $\diam J \le 2 i\varepsilon_n$, and
\item $\lvert \mathcal{J}_i\rvert \le (i+n)^{M_n}$.
\end{multistmt}
\end{claim}

\begin{proof}
The intervals $I_0$, $I_1$, \ldots, $I_{M_n-1}$ cover $K$, and therefore it is easy to see that the intervals in the system
\[\mathcal{J}_i = \{I_{\ell_0}\cdot I_{\ell_1} \cdot \ldots\cdot I_{\ell_{i-1}}: \text{$0\le \ell_j < M_n$ for each $0\le j < n$}\}\]
cover the $i$-fold Minkowski product $\ad{K}{i}$. As $\diam I_j = 2\varepsilon_n$ for $0\le j <M_n$, \autoref{diamprod} states that the length of these intervals is at most $2i\varepsilon_n$.

To prove (iii), notice that 
\[\mathcal{J}_i = \{I_{\ell_0}\cdot I_{\ell_1} \cdot \ldots\cdot I_{\ell_{i-1}}: 0\le \ell_0\le \ell_1 \le \ldots \le \ell_{i-1} < M_n\},\]
because the Minkowski product is commutative and thus we can freely reorder the terms in these products. It is a simple result in combinatorics that the number of integer sequences $(\ell_j)_{j=0}^{i-1}$ such that $0\le \ell_0\le \ell_1\le \ldots \le \ell_{i-1}< M_n$ is $\binom{i+M_n-1}{M_n-1}$, and therefore
\begin{align*}
\lvert \mathcal{J}_i\rvert &\le \binom{i+M_n-1}{M_n-1} = \frac{(i+1)\cdot (i+2)\cdot \ldots\cdot (i+M_n-1)}{(M_n-1)!} \le \\
&\le (i+M_n-1)^{M_n-1} \le (i+n)^{M_n},
\end{align*}
using the assumption \eqref{mnbound}.
\end{proof}

Now we are ready to construct the interval system which will cover $\ap{f}{K}$:
\begin{claim}\label{coverseries}
$\ap{f}{K}$ can be covered by at most $(2n)^{n\cdot M_n}$ intervals, each of length less than $4R_n$.
\end{claim}
\begin{proof}
Consider the interval systems $\mathcal{J}_1$, $\mathcal{J}_2$, \ldots, $\mathcal{J}_{n-1}$ guaranteed by \autoref{coverprod} and let $\mathcal{J}$ denote the system of all intervals which can be written as
\[a_0 + \sum_{i=1}^{n-1} a_i J_i + [-R_n, R_n]\quad\text{where $J_i\in \mathcal{J}_i$ for each $1\le j<n$.}\]
It is clear that the interval system $\mathcal{J}$ cover $\ap{f}{K}$, because $\mathcal{J}_i$ covers $\ap{K}{i}$ (for each $1\le i<n$) and \eqref{finiteexp} states that
\[\ap{f}{K} \subseteq \sum_{i=0}^{n-1} a_i \ad{K}{i} + [-R_n, R_n].\]

Applying \autoref{diamtriv} and \eqref{epsnchoice}, if $J\in\mathcal{J}$, then
\[\diam J \le 2 R_n + \sum_{i=1}^{n-1} \lvert a_i \rvert \cdot 2 i \varepsilon_n \le 2 R_n + \sum_{i=0}^\infty \lvert a_i \rvert \cdot 2 n \varepsilon_n = 2 R_n + 2n \varepsilon_n R_0 < 4 R_n.\]

Finally, for each $1\le i< n$ we know that $\lvert \mathcal{J}_i\rvert \le (i+n)^{M_n} < (2n)^{M_n}$ and therefore we can estimate the cardinality of $\mathcal{J}$ as
\[\lvert \mathcal{J}\rvert\le \prod_{i=1}^{n-1}\lvert \mathcal{J}_i\rvert \le (2n)^{n M_n}.\qedhere\]
\end{proof}

This cover of $\ap{f}{K}$ shows that
\[\mathcal{H}^s_\infty(\ap{f}{K}) \le (2n)^{n M_n} \cdot (4R_n)^s.\]

Taking logarithms on both sides yields that
\[\log \mathcal{H}^s_\infty(\ap{f}{K}) \le n\log (2n) \cdot  M_n + s\cdot\log (4R_n)\]
(with the convention that $\log 0 = -\infty$). Dividing both sides by $n\log (2n)$ yields
\[\frac{\log \mathcal{H}^s_\infty(\ap{f}{K})}{n\log (2n)} < M_n + s\cdot \frac{\log (4R_n)}{n\log (2n)}.\]
These calculations have demonstrated that this inequality holds for infinitely many $n\in\mathbb{N}$ (that is, for those $n\in\mathbb{N}$ where $K\in\mathcal{U}_n$). In \eqref{mnchoice} we have assumed that as $n$ converges to infinity, the right hand side of the inequality converges to $-\infty$. This is only possible if $\log \mathcal{H}^s_\infty(\ap{f}{K})=-\infty$, i.e. $\mathcal{H}^s_\infty(\ap{f}{K})=0$.
\end{proof}

\section{The case of large images}

In this section we will prove the following result:

\begin{theorem}\label{largemain}
Consider a power series $f(x) = \sum_{i=0}^\infty a_i x^i$ which converges absolutely on $[0,1]$. If there exists an $\varepsilon >0$ such that $\frac{\lvert a_{i+1}\rvert}{\lvert a_i\rvert}>\varepsilon$ for all but finitely many $i\in\mathbb{N}$, then $\Int \ap{f}{K}\neq \emptyset$ for the typical compact set $K\in\mathcal{K}([0,1])$. 
\end{theorem}

For example, this theorem implies the following:

\begin{example}
If $\lvert \lambda \rvert <1$, then $f(x) = \sum_{i=0}^\infty \lambda^i x^i = \frac{1}{1-\lambda x}$ converges absolutely on $[0,1]$ and $\Int \ap{f}{K} \neq \emptyset$ for the typical compact set $K\in\mathcal{K}([0,1])$.
\end{example}

Before proving \autoref{largemain}, we will state some simple lemmas about the following notion:

\begin{definition}\label{gapdef}
For $A \in \mathcal{K}(\mathbb{R})$, let
\[\gap A = \min \{\ell \ge 0 : A + [0,\ell]\text{ is an interval}\}.\]
\end{definition}

\noindent In other words, $\gap A$ is the length of the longest (bounded) interval complementary to $A$.

\begin{lemma}\label{gapab}
If $A, B\in\mathcal{K}(\mathbb{R})$ and $\gap A\le\diam B$, then $\gap(A+B)\le\gap B$.
\end{lemma}

\begin{proof}
\autoref{gapdef} implies that $B+[0, \gap B]$ is an interval. As the length of this interval is obviously at least $\diam B$ and $\diam B \ge \gap A$, we conclude that $A+B+[0, \gap B]$ is also an interval. This shows that
\[\gap B \ge \min \{\ell \ge 0 : A + B + [0,\ell]\text{ is an interval}\} = \gap (A+B).\qedhere\]
\end{proof}

\begin{lemma}\label{gapai}
Let $A_i \in \mathcal{K}(\mathbb{R})$ for $i\in\mathbb{N}$ and assume that $A = \sum_{i\in\mathbb{N}}A_i$ exists. If $\gap A_i \le\diam A_{i+1}$ for each $i\in\mathbb{N}$, then $A$ is an interval.
\end{lemma}

\begin{proof}
A straightforward proof by induction shows that 
\[\gap\left(\sum_{j=0}^{i} A_j\right)\le \gap A_i\quad\text{for each }i\in\mathbb{N}\]
(the inductive step is just \autoref{gapab}).

The existence of $A = \sum_{i\in\mathbb{N}}A_i\in \mathcal{K}(\mathbb{R})$ clearly implies that
\[\lim_{i\to \infty} \diam A_i = 0\quad\text{and therefore}\quad\lim_{i\to\infty} \gap A_i = 0.\]

It is easy to verify that $\gap : \mathcal{K}(\mathbb{R}) \to \mathbb{R}$ is a continuous map and therefore
\[0\le \gap A = \lim_{i\to\infty}\gap\left(\sum_{j=0}^{i} A_j\right) \le \lim_{i\to\infty} \gap A_i = 0.\]
It is clear from the definition that $\gap A=0$ implies that $A$ is an interval.
\end{proof}

This immediately implies the following:
\begin{corollary}\label{gapaiext}
Let $A_i \in \mathcal{K}(\mathbb{R})$ for $i\in\mathbb{N}$ and assume that $A = \sum_{i\in\mathbb{N}}A_i$ exists. If $\gap A_i \le\diam A_{i+1}$ for all but finitely many $i\in\mathbb{N}$, then $\Int A\neq\emptyset$.
\end{corollary}
\begin{proof}
If $\gap A_i \le\diam A_{i+1}$ for each $i \ge i_0$, then \autoref{gapai} states that
\[A'=\sum_{i=i_0}^\infty A_i\quad\text{is an interval.}\]
But then it is clear that 
\[A= A_0+A_1+\ldots+A_{i_0-1}+ A'\]
has nonempty interior, as we claimed.
\end{proof}

For the typical compact set $K$, the set $\ap{f}{K}=\sum_{i=0}^\infty a_i \ad{K}{i}$ is expressed as an infinite Minkowski sum. Unfortunately, we cannot apply the previous corollary directly, as $\gap(a_i \ad{K}{i})\nleq\diam(a_{i+1} K^{i+1})$ may happen \enquote{often} even in the case when the convergence $a_i\to 0$ is not too fast.

However, we can solve this by picking a certain subset $K'\subseteq K$, applying \autoref{gapaiext} for the infinite Minkowski sum $\ap{f}{K'} = \sum_{i=0}^\infty a_i\cdot \ad{K'}{i}$, and finally using that $\ap{f}{K'}\subseteq \ap{f}{K}$.

\begin{claim}\label{fpq}
Consider a power series $f(x) = \sum_{i=0}^\infty a_i x^i$ which converges absolutely on $[0,1]$ and satisfies that there exists an $\varepsilon >0$ such that $\frac{\lvert a_{i+1}\rvert}{\lvert a_i\rvert}>\varepsilon$ for all but finitely many $i\in\mathbb{N}$. If $0<p <q\le 1$ and $\frac{q-p}{q^2} < \varepsilon$, then $\Int \ap{f}{\{p, q\}} \neq \emptyset$.
\end{claim}

\begin{proof}
First we note that the absolute convergence of $f$ guarantees that $\ap{f}{\{p,q\}}$ is well-defined.

Let $A_i = a_i \cdot \ad{\{p, q\}}{i}$. It is easy to verify that
\[\diam A_i = \lvert a_i\rvert\cdot (q^i - p^i)\quad\text{and}\quad\gap A_i = \lvert a_i \rvert\cdot (q-p)\cdot q^{i-1}.\]
We know that $\frac{p^{i+1}}{q^{i+1}} \to 0$ and all but finitely many  $i\in\mathbb{N}$ satisfies that $\frac{a_{i+1}}{a_i}>\varepsilon>\frac{q-p}{q^2}$, therefore
\[\liminf_{i\to\infty}\frac{\diam A_{i+1}}{\gap A_i} =\liminf_{i\to\infty} \frac{\lvert a_{i+1} \rvert\cdot q^2}{\lvert a_i \rvert \cdot (q-p)}\left(1-\frac{p^{i+1}}{q^{i+1}}\right) > 1.\]
This means that $\gap A_i \le\diam A_{i+1}$ for all but finitely many $i\in\mathbb{N}$, but then we may use \autoref{gapaiext} to conclude that $\Int \ap{f}{\{p,q\}}\neq\emptyset$.
\end{proof}

Now we are ready to prove \autoref{largemain}:

\begin{restatelargemain}
Consider a power series $f(x) = \sum_{i=0}^\infty a_i x^i$ which converges absolutely on $[0,1]$. If there exists an $\varepsilon >0$ such that $\frac{\lvert a_{i+1}\rvert}{\lvert a_i\rvert}>\varepsilon$ for all but finitely many $i\in\mathbb{N}$, then $\Int \ap{f}{K}\neq \emptyset$ for the typical compact set $K\in\mathcal{K}([0,1])$. 
\end{restatelargemain}

\begin{proof}
It is well-known that the typical compact set $K\in\mathcal{K}([0,1])$ is perfect. In particular, this implies that there is a $0<r<1$ such that $K$ has infinitely many points in every neighbourhood of $r$.

If we pick $p, q\in K$ from a sufficiently small neighbourhood of $r$ (with $p<q$), then it is clear that they will satisfy $0<p<q\le 1$ and $\frac{q-p}{q^2}<\varepsilon$. Now we can apply \autoref{fpq} to see that $\Int \ap{f}{\{p,q\}}\neq \emptyset$ and therefore also $\Int \ap{f}{K}\neq \emptyset$, concluding the proof.
\end{proof}

\section{Open questions}\label{sec:last}

After examining the case of polynomials, the paper \cite{Chen} proposes the case of the exponential function
\[\exp(x)=\sum_{i=0}^\infty \frac{x^i}{i!},\]
which is one of the simplest examples of a nontrivial power series:
\newcommand{\citequestion}{{\cite[Question 4.3]{Chen}}}
\begin{question}[\citequestion]
Is it true that the typical compact set $K\subseteq [0,1]$ satisfies that $\dim_\mathrm{H} \ap{\exp}{K} = 0$?
\end{question}

The coefficients $a_i =\frac{1}{i!}$ lie in the gap where we cannot apply either \autoref{zerodimimg} or \autoref{largemain}. Therefore this question remains open, but our results underline its importance, showing that $\exp$ lies in the gap where the dimension of the image changes from zero to one.

Our results show that the image of the typical compact set may be zero-di\-men\-sional and may be one-dimensional for certain large classes of power series. It would be interesting to know whether all power series fall into these two classes:
\begin{question}
Is there a dimension $0<d<1$ and a power series $f$ that converges absolutely on $[0,1]$ such that the typical compact set $K\subseteq [0,1]$ satisfies that $\dim_\mathrm{H} \ap{f}{K} = d$?
\end{question}

Moreover, the statement of \autoref{largemain} leaves the following question open:
\begin{question}
Is there a power series $f$ that converges absolutely on $[0,1]$ such that the typical compact set $K\subseteq [0,1]$ satisfies that $\dim_\mathrm{H} \ap{f}{K} = 1$, but $\Int \ap{f}{K}=\emptyset$?
\end{question}

Also note that while we speak of \enquote{the} typical compact set, it is not evident that the dimension of the image of the typical compact set is meaningful in every case:
\begin{question}
Is there a power series $f$ that converges absolutely on $[0,1]$ such that there is no dimension $d$ for which the typical compact set $K\subseteq [0,1]$ satisfies that $\dim_\mathrm{H} \ap{f}{K} = d$?
\end{question}

\textsc{Acknowledgments.} We are grateful to M\'arton Elekes for many helpful remarks and discussions, and to Rich\'ard Balka for suggesting an improvement in \autoref{zerodimimg}.

\end{document}